\documentclass[12pt]{amsart}

\usepackage[utf8]{inputenc}

\usepackage{amssymb, graphicx}

\usepackage{amsfonts}
\usepackage{amsmath}
\usepackage{latexsym}
\usepackage{amscd}
\usepackage{verbatim}

\input xy
\xyoption{arrow} \xyoption{matrix}

\def\C{\mathbb{C}}

\usepackage[active]{srcltx}
\usepackage{enumerate,amsthm,amssymb,latexsym}
\usepackage[all]{xy}
\usepackage[parfill]{parskip}
\usepackage{amsmath,amssymb,xspace,amsthm}
\newtheorem{theorem}{Theorem}
\newtheorem{lemma}[theorem]{Lemma}
\newtheorem{proposition}[theorem]{Proposition}
\newtheorem{corollary}[theorem]{Corollary}
\numberwithin{equation}{section}
\newtheorem{definition}[theorem]{Definition}

\usepackage{amsthm}
\usepackage{multirow}
\usepackage{marginnote}

\renewcommand{\l}{\ensuremath{\lambda}}

\newcommand{\Z}{\ensuremath{\mathbb{Z}}\xspace}
\newcommand{\N}{\ensuremath{\mathbb{N}}\xspace}

\newcommand{\supp}{\ensuremath{\operatorname{Supp}}\xspace}

\renewcommand{\phi}{\varphi}

\renewcommand{\leq}{\leqslant}
\renewcommand{\geq}{\geqslant}

\def\ms{\mathfrak{s}}

\def\mh{\mathfrak{h}}

\def\sl{\mathfrak{sl}}

\def\D{\Delta}

\def\l{\lambda}

\def\Hom{\text{Hom}}

\def\C{\Bbb C}
\def\Z{\Bbb Z}

\def\c{\dot c}
\def\z{\dot z}

\begin{document}
\title{BGG category for the quantum Schr{\"o}dinger algebra}
\author{Genqiang Liu, Yang Li}
\maketitle

\begin{abstract} In this paper, we study the BGG category $\mathcal{O}$ for the quantum Schr{\"o}dinger algebra $U_q(\mathfrak{s})$, where
$q$ is a nonzero complex number which is not a root of unity.
If  the central charge
$\dot z\neq 0$, using the module $B_{\dot z}$ over the quantum Weyl algebra $H_q$,
we show that there is an equivalence between the full subcategory $\mathcal{O}[\z]$  consisting of modules with the central charge $\dot z$
and the BGG category  $\mathcal{O}^{(\mathfrak{sl}_2)}$  for  the quantum group $U_q(\mathfrak{sl}_2)$.  In the case that
$\dot z=0$, we study the subcategory $\mathcal{A}$  consisting of  finite dimensional $U_q(\mathfrak{s})$-modules of type $1$ with
zero action of  $Z$.  Motivated by the ideas in \cite{DLMZ, Mak},
we directly construct an equivalent functor from $\mathcal{A}$  to the category of finite
dimensional representations of an infinite quiver with some quadratic
 relations.  As a corollary,
we show that   the category of finite dimensional $U_q(\mathfrak{s})$-modules is wild.
\end{abstract}

\section{Introduction}
In this paper, we denote by $\mathbb{Z}$, $\mathbb{Z}_+$, $\mathbb{N}$,
$\mathbb{C}$ and $\mathbb{C}^*$ the sets of  all integers, nonnegative integers,
positive integers, complex numbers, and nonzero complex numbers, respectively.
Let $q$ be a nonzero complex number which is not a root of unity. For $n\in \mathbb{Z}$, denote
$[n]=\frac{q^n-q^{-n}}{q-q^{-1}}$.

The BGG  category $\mathcal{O}$   for complex semisimple Lie algebras was
introduced by  Joseph Bernstein, Israel Gelfand and Sergei Gelfand in the early 1970s, see \cite{BGG},  which includes all highest weight modules
such as Verma modules and finite dimensional simple modules.
This category is influential in many areas of representation theory.  About the knowledge of  $\mathcal{O}$, one can see the
recent monograph \cite{Hu}  for details.

The Schr\"{o}dinger Lie algebra $\ms$ is the
 semidirect product of $\mathfrak{sl}_2$ and the three-dimensional Heisenberg Lie algebra.
 This algebra can describe symmetries of the free particle Schr\"{o}dinger equation,  see \cite{DDM1,Pe}. The representation theory of the Schr\"{o}dinger algebra  has been studied by  many authors.   A classification of the simple
highest weight representations of the Schr\"{o}dinger algebra were given in \cite{DDM1}.
All simple weight modules with  finite dimensional weight spaces were
classified in \cite{D}, see also \cite{LMZ}.    All simple weight modules of the Schr\"{o}dinger algebra were
classified in \cite{BL2,BL3}.  The BGG category
$\mathcal{O}$ of $\ms$  was studied in \cite{DLMZ}.

In 1996, in order to research the $q$-deformed heat equations,
a $q$-deformation of the universal enveloping algebra of the Schr\"{o}dinger Lie algebra was
introduced in \cite{DDM2}. This algebra is called the quantum Schr{\"o}dinger algebra.

 The quantum Schr{\"o}dinger algebra  $U_q(\ms)$ over $\C$  is generated by the elements
  $E$, $F$, $K$, $K^{-1}$, $X$, $Y$, $Z$ with defining relations:
\begin{equation}
\label{commrelations}
\begin{array}{ll}
[E,F]=\frac{K-K^{-1}}{q-q^{-1}}, & KXK^{-1} =qX,  \\
KEK^{-1}=q^2 E,  &KFK^{-1}=q^{-2}F, \\
 KYK^{-1} =q^{-1}Y,  & qYX-XY=Z, \\
EX = qXE, & EY = X+ q^{-1}YE, \\
FX= YK^{-1}+XF, &FY=YF,
\end{array}
\end{equation}
where $Z$ is central in $U_q(\ms)$. This definition is somewhat different from   that of \cite{DDM2} in form.
This algebra is a kind of  quantized symplectic oscillator algebra of rank one, see \cite{GK}.   The paper \cite{GK}  gave the PBW
Theorem and showed that the category  $\mathcal{O} $ is a highest weight category.
In the present  paper, we will give specific characterizations for each block of
 $\mathcal{O}$ for  $U_q(\ms)$ using some quivers.

The subalgebra  generated by $E,F,K,K^{-1}$ is the quantum group $U_q(\sl_2)$. The
 subalgebra generated by  $X,Y,Z$ is called the quantum Weyl algebra $H_q$, see \cite{B}.
All simple weight modules  over $U_q(\ms)$  with zero action of $Z$ were
classified in \cite{BL1}.

The paper is organized as follows. In Section 2 we recall some basic
facts about the category  $\mathcal{O}$ for  $U_q(\ms)$ .  For a $\z \in \C$, we denote by $\mathcal{O}[\z]$ the full subcategory of
$\mathcal{O}$ consisting of all modules which  are annihilated by some power of the maximal ideal
$\langle Z-\z\rangle$ of $\C[Z]$.   In Section 3, in case  of $\z\neq 0$,  using the modules $B_{\z}$ over the quantum Weyl algebra $H_q$, we show that the
functor $-\otimes \tilde{B}_{\z}$ gives an equivalence between
$\mathcal{O}^{(\mathfrak{sl}_2)}$ and $\mathcal{O}[\z]$ , where $\mathcal{O}^{(\mathfrak{sl}_2)}$ is the BGG
 category of $U_q(\mathfrak{sl}_2)$,  see Theorem \ref{prop721}.
 A  Weight $U_q(\ms)$-module $M$ is of type $1$ if the $\supp(M)\subset q^\Z$.  In Section 4, we study  the
category $\mathcal{A}$ of finite dimensional  $U_q(\ms)$-modules of type $1$ with zero action of $Z$.  It is shown that
there is an equivalence  between $\mathcal{A}$  and the category of finite dimensional representations of an infinite quiver with some quadratic
 relations, see Theorem \ref{thm13}.  In \cite{DLMZ}, the  grading technique was used in the study of finite dimensional modules for the
Schr\"{o}dinger Lie algebra.


\section{Basic properties of the category $\mathcal{O}$}

\subsection{ The definition of Category $\mathcal{O}$}

Let $U_q(\mathfrak{n}^+)$ be the subalgebra of $U_q(\ms)$ generated by the elements $E,X$
and let $U_q(\mathfrak{n}^-)$ be the subalgebra generated by $F,Y$.
Moreover let $U_q(\mh)$
be the subalgebra generated by the elements $K, K^{-1}, Z$. We write $\otimes$ for $\otimes_\C$.

 Then we have the following triangular  decomposition:
\begin{equation}\label{eq1}
U_q(\ms)=U_q(\mathfrak{n}^-)\otimes U_q(\mh)\otimes U_q(\mathfrak{n}^+).
\end{equation}

A $U_q(\ms)$-module $V$  is  called a weight module if $K$ acts diagonally on $V$, i.e.,
$$V=\oplus_{\lambda\in \mathbb{C^*}}V_\lambda,$$ where $V_\lambda=\{v\in V \mid K v=\lambda v\}$.
For a weight module $V$, let $$\mathrm{supp}(V)=\{\lambda\in\C^*|V_\lambda\neq0\}.$$

Next, we introduce the category $\mathcal{O}$  for
$U_q(\ms)$.
\begin{definition}
A left module $M$ over $U_q(\ms)$ is said to belong to category $\mathcal{O}$  if
\begin{enumerate}[(1)]
\item $M$ is finitely generated over $U_q(\ms)$;
\item $M$ is a weight module;
\item The action of $U_q(\mathfrak{n}^+)$ on $M$ is locally finite, i.e., $\dim U_q(\mathfrak{n}^+) v< \infty$ for any
$v\in M$.
\end{enumerate}
\end{definition}

For a weight module $V$,  a weight vector $v_\lambda\in V_\lambda$ is called a {\it highest weight vector} if $E v_\lambda= Xv_\lambda =0$.
A module  $M$ is
called a highest weight module of highest weight $\lambda$ if there exists a highest weight vector
$v_\lambda$ in $M $ which generates $M$.
For $\lambda\in\C^*$ and $\z\in\C$, let $\Delta(\lambda,\z)$ be the Verma module generated by $v_\lambda$,
where $Kv_\l=\lambda v_\l, Zv_\l=\z v_\l$. Then $\{Y^kF^lv_\l|k,l\in\Z_+\}$  is a basis of $\Delta(\lambda,\z)$.
Let $R(\lambda,\z)$ be the largest proper submodule of $\Delta(\lambda,\z)$. Hence $L(\lambda,\z)=\Delta(\lambda,\z)/R(\lambda,\z)$ is the unique simple
quotient module of $\Delta(\lambda,z)$.

\subsection{Basic properties of $\mathcal{O}$}
By the similar arguments as those in \cite{Hu}, we can see that every module $M$ in $\mathcal{O}$  has the following standard properties.

\begin{lemma} The category  $\mathcal{O}$ is closed with respect to taking submodules, quotient modules and finite direct sums. That is,
the category  $\mathcal{O}$ is an abelian  category.
\end{lemma}

\begin{lemma}\label{basic-property} Let $M$ be any module in $\mathcal{O}$.
\begin{enumerate}[$($1$)$]
\item The module $M$ has a finite filtration
$$0 = M_0 \subset  M_1\subset \cdots  \subset M_n = M$$
such that each subquotient $M_j/M_{j-1}$ for $1 \leq j\leq n$ is a highest weight module.
\item  Each weight space of $M$  is finite dimensional.
\item Any simple module in  $\mathcal{O}$ is isomorphic to some $L(\lambda,\z)$, for $\l, \z\in \C$.
\end{enumerate}
\end{lemma}

\section{Blocks of nonzero central charge}
In this section, we assume that $\z\neq 0$.  We denote by $\mathcal{O}[\z]$ the full subcategory of
$\mathcal{O}$ consisting of all modules which are annihilated by some power of the maximal ideal
$\langle Z-\z\rangle$ of $\C[Z]$. Let $\mathcal{O}^{(\mathfrak{sl}_2)}$ denote the BGG category for $U_q(\mathfrak{sl}_2)$.
We will show that there is an equivalence between
$\mathcal{O}^{(\mathfrak{sl}_2)}$ and $\mathcal{O}[\z]$ . Firstly,  we found that the structure of Verma  modules over $U_q(\ms)$
is similar as that of  Verma  modules over $U_q(\sl_2)$.

\subsection{The tensor product realizations of highest modules}

In this subsection, we will give tensor product realizations  of Verma modules using Verma modules over $U_q(\sl_2)$ and $H_q$.
This construction is crucial to the study of the category $\mathcal{O}[\z]$ .

For a nonzero $\z\in \C$,
let $$B_{\z}:=H_q/(H_q(Z-\z)+H_qX)$$ which is a simple $H_q$-module. Denote the image of $Y^i$ in $B_{\z}$ by $v_i$ for $i\in \Z_+$. We can see that
\begin{equation}\label{xaction}X v_i=-\frac{\z(q^i-1)}{q-1}v_{i-1}, i\in\Z_+.\end{equation}
Define the action of $U_q(\sl_2)$ on $B_{\z}$ by
\begin{equation}\aligned
          K v_i&= q^{-\frac{1}{2}-i}v_i, \\
F v_i&=\frac{q^{\frac{1}{2}}}{\z(q+1)}v_{i+2}, \\
E v_i&= \frac{-\z(q^i-1)(q^{i-1}-1)}{q^{i-2}(q^2-1)(q-1)}v_{i-2}.\label{module}
\endaligned
\end{equation}
Then we can check that
$$\aligned
(EF-FE)v_i&=\frac{K-K^{-1}}{q-q^{-1}} v_i,\\
EYv_i &= Xv_i+ q^{-1}YEv_i, \\
FXv_i&= YK^{-1}v_i+XFv_i.
\endaligned
$$
Thus the action (\ref{module}) indeed makes $B_{\z}$  to be a module over $U_q(\ms)$. We denote this $U_q(\ms)$-module by
$\widetilde{B}_{\z}$.  In fact, $\widetilde{B}_{\z}\cong L(q^{-\frac{1}{2}},\z)$.

We can make
a $U_q(\sl_2)$-module $N$  to  be a $U_q(\ms)$-module be defining $H_q N=0$. We denote the resulting
$U_q(\ms)$-module by $\widetilde{N}$.

\begin{lemma}
The following map
\begin{equation}\label{comu}\Delta: U_q(\ms)\rightarrow U_q(\ms)\otimes U_q(\ms)\end{equation}defined by
$$\aligned\Delta(E)&=1\otimes E +E\otimes K, & \Delta(F)&=K^{-1}\otimes F +F\otimes 1\\
\Delta(K)&=K\otimes K, &\Delta(K^{-1})&=K^{-1}\otimes K^{-1},\\
\Delta(X)&=1\otimes X , &  \Delta(Y)&=1\otimes Y.\\
\Delta(Z)&=1\otimes Z
\endaligned
$$ can define  an algebra homomorphism.
\end{lemma}

\noindent{\bf Remark:} There is no algebra homomorphism $\epsilon:U_q(\ms)\rightarrow \C$ such that the following diagram commutes:
\begin{displaymath}
\xymatrix{
  U_q(\ms) \ar[d]_{id} \ar[r]^{\Delta}
                & U_q(\ms)\otimes U_q(\ms) \ar[d]^{1\otimes\epsilon}  \\
  U_q(\ms) \ar[r]_{can}
                & U_q(\ms)\otimes \C           }
\end{displaymath}
So we can not define a bialgebra structure on $U_q(\ms)$ from $\Delta$.

 Via the map $\Delta$, the space $\widetilde{N}\otimes  \widetilde{B}_{\z}$ can be defined as
a $U_q(\ms)$-module for any $U_q(\sl_2)$-module $N$. More precisely,  for $u\in U_q(\ms)$, if
$\Delta(u)=\Sigma_i u_i\otimes w_i$, then the action of $U_q(\ms)$ on $\widetilde{N}\otimes  \widetilde{B}_{\z}$ is defined by
$$ u(n\otimes b )=\Sigma_i u_i n\otimes w_ib,\ \ \  n\in N, b\in  \widetilde{B}_{\z}.$$

Let $\D_{\sl_2}(\lambda )$ be the Verma module over $U_q(\sl_2)$
with the highest weight $\lambda $ whose unique simple quotient module is $L_{\sl_2}(\lambda )$.   It is well known that the module
 $\D_{\sl_2}(\lambda )$ is reducible  if and only if  $\lambda^2\in q^{2\Z_+}$, see \cite{J}.
 For each $d\in\Z_+$, we have  non-split
short exact sequences
\begin{displaymath}
0\to \Delta_{\sl_2}(q^{-d-2})\to \Delta_{\sl_2}(q^{d}) \to L_{\sl_2}(q^{d})\to 0,
\end{displaymath}
and
\begin{displaymath}
0\to \Delta_{\sl_2}(-q^{-d-2})\to \Delta_{\sl_2}(-q^{d}) \to L_{\sl_2}(-q^{d})\to 0.
\end{displaymath}

The structure of  $\D(\lambda,\z)$  was determined  in \cite{DDM2}.  The following proposition give a constructive
proof.
\begin{proposition}\cite{CCL}\label{p3}
The following results hold.
\begin{enumerate}
\item   If $\z\neq 0$, then  $\D(\lambda,\z)\cong \widetilde{\D}_{\sl_2}(\lambda q^{\frac{1}{2}})\otimes \widetilde{B}_{\z}$.
Therefore the Verma  module $\Delta(\lambda,\z)$ is reducible  if and only if  $\lambda^2\in q^{2\Z_+-1}$.  Moreover, $L(\lambda,z)\cong \widetilde{L}_{\sl_2}(\lambda q^{\frac{1}{2}})\otimes \widetilde{B}_{\z}$.
\item For each $d\in\Z_+$, we have a non-split
short exact sequences
\begin{displaymath}
0\to \Delta(q^{-d-\frac{5}{2}},\z)\to \Delta(q^{d-\frac{1}{2}},\z) \to L(q^{d-\frac{1}{2}},\z)\to 0,
\end{displaymath}
and
\begin{displaymath}
0\to \Delta(-q^{-d-\frac{5}{2}},\z)\to \Delta(-q^{d-\frac{1}{2}},\z) \to L(-q^{d-\frac{1}{2}},\z)\to 0.
\end{displaymath}
\end{enumerate}
\end{proposition}
%

\subsection{ Equivalence between
$\mathcal{O}^{(\mathfrak{sl}_2)}$ and $\mathcal{O}[\z]$ .}

\begin{lemma}\label{lemma7} If $\z\neq 0$, then any module in $\mathcal{O}[\z]$ has finite composition length.
\end{lemma}
\begin{proof} According  to (1) in Lemma \ref{basic-property}, any nonzero module $M$ in $\mathcal{O}[\z]$ has a finite
filtration with sub-quotients given by highest weight modules. Hence it suffices to treat the case that $M$ is the Verma module
$\D(\lambda,\z)$. By Proposition \ref{p3}, $\D(\lambda,\z)$ has finite composition length if $\z\neq 0$.
\end{proof}

\begin{proposition}\label{mainlemma}
Suppose that $ V$ is a module in $\mathcal{O}[\z]$ with nonzero central
charge $\dot{z}$.  Then $V\cong \widetilde{N}\otimes  \widetilde{B}_{\z}$ for some
$U_q(\sl_2)$-module $N$.
\end{proposition}
\begin{proof} By Lemma \ref{lemma7}, $V$ has a finite composition length $l.$  We will proceed the proof by induction on $l$.
Firstly, we consider the case $l=1$. The fact  that $V\in \mathcal{O}$ forces that $V$ is a simple highest weight  $U_q(\ms)$-module.
Then $V$ is a simple quotient module of some  Verma module  $\D(\lambda,z)$ . Note that $ \D(\lambda,z) \cong \widetilde{\D}_{\sl_2}(\lambda q^{\frac{1}{2}})\otimes \widetilde{B}_{\z}$. Thus $V\cong \widetilde{L}_{\sl_2}(\lambda q^{\frac{1}{2}})\otimes \widetilde{B}_{\z}$.

Next, we consider the general case. Let $M$ be a maximal submodule of $V$.  By the induction hypothesis, we see that
$$M\cong \widetilde{N_1}\otimes  \widetilde{B}_{\z}, \qquad V/M \cong \widetilde{N_2}\otimes  \widetilde{B}_{\z},$$ where $N_1, N_2$ are
$U_q(\sl_2)$-modules.

As vector spaces,  we can assume that  $V=N\otimes {B}_{\z}$,  where $N$ is a vector space such that $N_1\subseteq N$ and $N/N_1\cong N_2$.
Moreover, $u(w\otimes v)=w\otimes uv, K(w\otimes v)= (Kw)\otimes(K v)$,  for $u\in H_q, w\in N, v\in B_{\z}$.

For $w\in N$,  we can find   $w_i\in N, 0 \leq i \leq k$ such that~$E( w\otimes v_0)=\sum_{i=0}^k w_i\otimes v_i$, where
$v_i$  is the image of $Y^i$ in $B_{\z}$ .  From $X^kE=q^{-k}EX^k, Xv_0=0$, we have
$$X^kE( w\otimes v_0)=w_k\otimes X^k v_k=q^{-k}( w\otimes X^k v_0)=0. $$  By (\ref{xaction}), when $k>0$,  $X^k v_k\neq 0$ . So we must have that $k=0$.  Denote $w_E=q^{\frac{1}{2}}w_0$. Then
 $E( w\otimes v_0)=w_E\otimes K v_0$.

 From $$E v_i= \frac{-\z(q^i-1)(q^{i-1}-1)}{q^{i-2}(q^2-1)(q-1)}v_{i-2}$$ and
\begin{equation}
EY^i=q^{-i}Y^iE+[i]Y^{i-1}X-\frac{(q^i-1)(q^{i-1}-1)}{q^{i-2}(q^2-1)(q-1)}ZY^{i-2},\end{equation}
we obtain that
\begin{equation}\label{EA} \aligned & E(w\otimes v_i)=EY^i(w\otimes v_0)\\
=& \Big(q^{-i}Y^iE+[i]Y^{i-1}X-\frac{(q^i-1)(q^{i-1}-1)}{q^{i-2}(q^2-1)(q-1)}ZY^{i-2}\Big)(w\otimes v_0)\\
=& w_E\otimes K v_i + w\otimes E v_i.
\endaligned\end{equation}
 By the action of $F$ on $N/N_1\otimes \widetilde{B}_{\z}$, there exist $w_F\in N,  w'_i\in N$, $1\leq i \leq k$ satisfying
$$F( w\otimes v_0)=  w_F \otimes v_0+ K^{-1}w \otimes F v_0+\sum_{i=1}^k w'_i\otimes v_i, $$ where  $\sum_{i=1}^k w'_i\otimes v_i\in N_1\otimes \widetilde{B}_{\z}$.

From $$\aligned 0=& F( w\otimes Xv_0)=FX( w\otimes v_0)\\
=& XF( w\otimes v_0)+YK^{-1}( w\otimes v_0)\\
=& K^{-1}w \otimes X F v_0+\sum_{i=1}^k  w'_i\otimes X v_i+q^{\frac{1}{2}}(K^{-1}w)\otimes v_1\\
= & -q^{\frac{1}{2}}(K^{-1}w) \otimes  v_1- \frac{\z(q^i-1)}{q-1}\sum_{i=1}^k  w'_i\otimes  v_{i-1}+q^{\frac{1}{2}}(K^{-1}w)\otimes v_1\\
=& - \frac{\z(q^i-1)}{q-1}\sum_{i=1}^k  w'_i\otimes  v_{i-1},
\endaligned$$ we have $ w'_i=0$, for any $1\leq i \leq k.$
Then $$ F( w\otimes v_0)=  w_F \otimes v_0+ K^{-1}w \otimes F v_0.$$  Consequently,  from $YF=FY$ and $Y^iv_0=v_i$, we have
 \begin{equation}\label{FA}F( w\otimes v_i)=  w_F \otimes v_i+ K^{-1}w \otimes F v_i,\ \ i\in\Z_+. \end{equation}
We can define the action of $U_q(\sl_2)$  on $N$ as follows:
$$E\cdot w= w_E,\ \ F\cdot w= w_F,\ \ w\in N.$$

From (\ref{EA}) and (\ref{FA}), we can see that $V\cong \widetilde{N}\otimes  \widetilde{B}_{\z}$.
The proof is complete.
\end{proof}

By Lemma \ref{mainlemma}, we have the following category equivalence.
\begin{theorem}\label{prop721} If $\z\neq 0$,
using the algebra homomorphism $\Delta$ in (\ref{comu}), we can define a functor
\begin{displaymath}
-\otimes \tilde{B}_{\z}: \mathcal{O}^{(\mathfrak{sl}_2)}\to \mathcal{O}[\z].
\end{displaymath}
Moreover, this functor is an equivalence of categories.
\end{theorem}
\begin{proof}By the definition of category $ \mathcal{O}$, the functor $F:=-\otimes \tilde{B}_{\z}$ maps
modules in $\mathcal{O}^{(\mathfrak{sl}_2)}$ to modules in $ \mathcal{O}[\z]$.  By Lemma \ref{mainlemma}, the functor $F$
is essentially surjective.

Next,  we will show that for any $M,N\in \text{Obj}(\mathcal{O}^{(\mathfrak{sl}_2)}$),   the map $$F_{M,N}:\Hom_{\mathcal{O}^{(\mathfrak{sl}_2)}}(M,N)\rightarrow \Hom_{\mathcal{O}[\z]}(F(M),F(N))$$ is a bijection.

For $f\in \Hom_{\mathcal{O}^{(\mathfrak{sl}_2)}}(M,N)$ such that $F_{M,N}(f)=f\otimes 1=0$,
we must  have
 $f=0$. So  $F_{M,N}$ is injective.

For any $g\in \Hom_{\mathcal{O}[\z]}(F(M),F(N))$,  suppose that $$g(w\otimes v_0)=\sum_{i=0}^k w_i\otimes v_i, w\in M, w_k\in N,$$
where $v_i$ is the image of $Y^i$ in $ B_{\z}$ . Since $\tilde{B}_{\z}$ is a simple $H_q$-module,  by the density theorem, there exists
 ~$u\in H_q$ such that $$uv_0=v_0, uv_i=0,\ \  i=1,\dots, k.$$
  Form ~$$ug(w\otimes v_0)=g(u(w\otimes v_0))=g(w\otimes uv_0 ),$$  we have:
 $$u(\sum_{i=0}^k w_k\otimes v_k)=\sum_{i=0}^k w_k\otimes uv_k= w_0\otimes v_0=g(w\otimes v_0 ).$$
Define the map $f:M\rightarrow N$ such that $f(w)=w_0$, i.e., $g(w\otimes v_0 )=f(w)\otimes v_0.$

From $Y^ig(w\otimes v_0 )=g(w\otimes Y^iv_0 )$,  we have $$g(w\otimes v_i )=f(w)\otimes v_i, \ \ \forall\ i\in\Z_+.$$
By  $Eg(w\otimes v_0 )=g(E(w\otimes v_0))$ and $Ev_0=0$,  we have that
  $$\aligned Ef(w)\otimes K v_0=& f(w)\otimes Ev_0+ f(Ew)\otimes K v_0
  =&f(Ew)\otimes Kv_0,\endaligned$$
  i.e., $Ef(w)=f(Ew)$. Similarly, we can check that $Ff(w)=f(Fw)$, $Kf(w)=f(Kw)$.
  Then $f\in \Hom_{\mathcal{O}^{(\mathfrak{sl}_2)}}(M,N)$.   So $F_{M,N}(f)=g$, $F_{M,N}$ is surjective.

 Therefore $F_{M,N}$ is bijective, and hence  $F:=-\otimes \tilde{B}_{\z}$  is an equivalence.
\end{proof}

\subsection{ The description of $\mathcal{O}[\xi,\c,\z]$}

Let $C = FE+\frac{qK+q^{-1}K^{-1}}{(q-q^{-1})^2}$ be the Casimir element of $U_q(\sl_2)$. Define the following element in $U_q(\ms)$:
$$\tilde{C}=Z\big((1+q^{-1})C+\frac{1}{q^2-1}K^{-1}\big)+ X^2F-Y^2EK^{-1} + XY (q^{-1}FE-qEF).$$
In \cite{CCL}, the following lemma was proved.
\begin{lemma}The element $\tilde{C}$ belongs to the center of $U_q(\ms)$.
\end{lemma}

For a module  $M$  in $\mathcal{O}$,  from that the weight spaces of $M$ are finite dimensional, we can see that
the action of $\C[\tilde{C},Z]$ on $M$ is locally finite.

For $\lambda\in\C, \z\in\C$, let $v_\lambda$ be a highest vector of the $U_q(\ms)$-module $L(\lambda, \z)$ . We denote by $\tilde{c}_\lambda$ the scalar corresponding to the action
of the central element $\tilde{C}$ on $v_\lambda$.   Similarly, for the $U_q(\sl_2)$-module $L_{\sl_2}(\lambda)$, we denote scalar corresponding to the action
of $C$ by $c_{\lambda}$.

\begin{lemma}\label{lemma2} We have the following:
\begin{enumerate}[(1)]
\item $\tilde{c}_\lambda=\frac{\z}{(q-q^{-1})^2}\Big((q+q^2)\l+(q^{-1}+q^{-2})\l^{-1}\Big)$.
\item $\tilde{c}_\lambda=\tilde{c}_{\lambda q^{-k}}$ iff $ \l^2= q^{k-3}$.
\end{enumerate}
\end{lemma}

Let $\xi\in\C^*/q^\Z$. We denote by $\mathcal{O}[\xi]$ the full subcategory of
$\mathcal{O}$ consisting of all $M$ such that $\mathrm{supp}(M)\subset \xi$.
For $\z, \c\in \C$, we denote by $\mathcal{O}[\xi, \c,\z]$ the full subcategory of
$\mathcal{O}[\xi]$ consisting of all $M$ such that $M$ is annihilated by some power of the maximal ideal
$\langle \tilde{C}-\c,Z-\z\rangle$ of $\C[Z,\tilde{C}]$.  Since the action of $\C[\tilde{C},Z]$ on $M$ is locally finite,  we have

\begin{displaymath}
\mathcal{O}[\xi]\cong\bigoplus_{\c,\z\in\C}\mathcal{O}[\xi,\c,\z].
\end{displaymath}

Similarly,  we can define the subcategory $ \mathcal{O}^{(\mathfrak{sl}_2)}[\xi,\c] $  of $\mathcal{O}^{(\mathfrak{sl}_2)}$, where $\c$ is defined
by the action of
Casimir element  $C$ of $U_q(\sl_2)$.

Using the equivalence between
$\mathcal{O}^{(\mathfrak{sl}_2)}$ and $\mathcal{O}[\z]$  given in Theorem \ref{prop721}, we have the following
equivalence.
\begin{lemma} \label{blockequ}
The restriction functor
\begin{displaymath}
-\otimes \tilde{B}_{\z}: \mathcal{O}^{(\mathfrak{sl}_2)}[\xi, c_\lambda] \to \mathcal{O}[q^{-\frac{1}{2}}\xi, \tilde{c}_{\lambda q^{-\frac{1}{2}}},\z]
\end{displaymath}
 is an equivalence of categories, where $\lambda\in\xi, c_\lambda=\frac{q\lambda+q^{-1} \lambda^{-1} } {(q-q^{-1})^2 }$.
\end{lemma}
%
%

Using the structure of  $\mathcal{O}^{(\mathfrak{sl}_2)}$ (see Section 5.3 in \cite{Ma}),
we give  descriptions of each block $\mathcal{O}[\xi,\c,\z]$ as follows.

\begin{proposition}\label{half}
Let $\xi=q^{\frac{1}{2}+\mathbb{Z}},\z\in\C^*$.    Then the following claims hold.
\begin{enumerate}[$($1$)$]
\item The module $\Delta(q^{-\frac{3}{2}}, \z)$ is the unique simple object in $\mathcal{O}[q^{\frac{1}{2}+\mathbb{Z}}, \tilde{c}_{q^{-\frac{3}{2}}},\z]$.
Moreover, the block $\mathcal{O}[q^{\frac{1}{2}+\mathbb{Z}}, \tilde{c}_{q^{-\frac{3}{2}}},\z]$ is semisimple.
\item   For $n\in\Z_+$, $\Delta(q^{-n-\frac{5}{2}},\z)$ ,  $L(q^{n-\frac{1}{2}}, \z)$
are all the simple objects in $\mathcal{O}[q^{\frac{1}{2}+\mathbb{Z}}, \tilde{c}_{q^{n-\frac{1}{2}}},\z]$.
\item   For $n\in\Z_+$,  the subcategory  $\mathcal{O}[q^{\frac{1}{2}+\mathbb{Z}}, \tilde{c}_{q^{n-\frac{1}{2}}},\z]$ is
equivalent to the category of finite dimensional representations over $\mathbb{C}$ of
the following quiver with relations:
\begin{displaymath}\label{quiver}
\xymatrix{\bullet\ar@/^/[rr]^{a}&&\bullet\ar@/^/[ll]^{b}}\qquad ab=0.
\end{displaymath}
\end{enumerate}
\end{proposition}

\begin{proposition}\label{integral}
Let $\xi=q^{\mathbb{Z}}, \z\in\C^*$.    Then we have the following:
\begin{enumerate}[$($1$)$]
\item For any $n\in\Z$, the Verma module $\Delta(q^{n}, \z)$ is simple.
\item  The modules $\Delta(q^{n}, \z)$,  $\Delta(q^{2n+3}, \z)$ are  simple objects in $\mathcal{O}[q^{\mathbb{Z}}, \tilde{c}_{q^{n}},\z]$.
\item  The block $\mathcal{O}[q^{\mathbb{Z}}, \tilde{c}_{q^{n}},\z]$ is equivalent to $\C\oplus \C$-mod for any $n\in \Z$.
\end{enumerate}
\end{proposition}

\begin{proposition}
 For $\xi\in \C^*/q^\Z, \z\in\C^*, \lambda\in\xi$.   If $\xi\neq \pm q^{\mathbb{Z}}, \pm q^{\frac{1}{2}+\mathbb{Z}}$, then
the block $\mathcal{O}[\xi, \tilde{c}_\lambda,\z]$ is semisimple with the unique simple object $\Delta(\lambda,\z)$ .
\end{proposition}

\section{Finite dimensional $U_q(\ms)$-modules}

In this section, we study finite dimensional $U_q(\ms)$-modules.
A  Weight $U_q(\ms)$-module $M$ is of type $1$ if the $\supp(M)\subset q^\Z$.
Note that $L(\lambda, \z)$ is infinite dimensional when $\z\neq 0$. So $Z$ acts trivially on any finite
dimensional simple  $U_q(\ms)$-module.  Consequently  $Z$ acts nilpotently  on any finite
dimensional  $U_q(\ms)$-module.
  Let $\mathcal{A}$ denote the
category of finite dimensional  $U_q(\ms)$-modules of type $1$ with zero action of $Z$.  Thus any module in  $\mathcal{A}$ is a module over the smash product algebra
$A:=\C_q[X,Y]\rtimes U_q(\sl_2)$,  where $\C_q[X,Y]=\C\langle X,Y| XY=qYX \rangle$ is the quantum plane.
The quantum plane $\C_q[X,Y]$  is a $U_q(\sl_2)$-module on the following action.
\begin{equation}
\label{commrelations2}
\begin{array}{lll}
 K\cdot X= qX& E\cdot X=0 & F\cdot X=Y,   \\
 K\cdot Y= q^{-1}Y& E\cdot Y=X & F\cdot Y=0.
\end{array}
\end{equation}

We will  use the completely reducibility of  finite dimensional $U_q(\sl_2)$-modules and the  Clebsch-Gordon rule to discuss  the category
$\mathcal{A}$ .  For convenience,  let $L(i)$ denote the finite dimensional $U_q(\sl_2)$-module with the highest weight  $q^i, i\in\Z_+$.  In fact, we can assume that $L(1)=\C X+ \C Y$,
whose $U_q(\sl_2)$-module structure was defined by (\ref{commrelations2}).
It is well known that  the tensor product $L(m)\otimes L(n)$ is a $U_q(\sl_2)$-module under the
 action defined by the following co-multiplication: $$\aligned
\Delta'(E)&=E\otimes 1 +K\otimes E, \quad  \quad \Delta'(F)=F\otimes K^{-1} +1\otimes F,\\
\Delta'(K)&=K\otimes K, \quad \quad \quad\quad \quad \Delta'(K^{-1})=K^{-1}\otimes K^{-1}.\\
\endaligned
$$
\noindent{\bf Remark: } The above co-multiplication $\Delta'$~is different from $ \Delta$ in (\ref{comu}).  Now $\Delta'$ can guarantee that
 $\tau(\theta_i)$ defined in Lemma \ref{LA1}  is a $U_q(\sl_2)$-module homomorphism, however $ \Delta$ in (\ref{comu}) can not. Because we consider the
 left action of $L(1)$ on $L(i)$ in Lemma \ref{LA1} .

Next, we will introduce two lemmas which are used in the proof of  Theorem \ref{thm13}.
\begin{lemma}\label{zuigaoquan}\begin{enumerate}[$($1$)$]
\item  $L(1)\otimes L(i)\cong L(i+1)\oplus L(i-1)$, for $i\geq 1$;
\item Suppose that $v_i$ is a highest weight vector of $L(i)$. Then $X\otimes v_i$ is a  highest weight vector of $L(1)\otimes L(i)$ whose highest weight  is $q^{i+1}$,
and $[i] Y\otimes v_i- q^{-1}X\otimes Fv_i  $ is a  highest weight vector of $L(1)\otimes L(i)$ whose highest weight  is $q^{i-1}$;
\item In $L(1)\otimes L(1)\otimes L(i)$,  the elements
 $$ [i+1]Y\otimes X\otimes v_i-q^{-1}X\otimes X\otimes Fv_i-q^{-i-1}X\otimes Y \otimes v_i $$
  and $$[i] X\otimes Y\otimes v_i- q^{-1}X\otimes X\otimes Fv_i $$
are highest weight vectors with the  highest weight $q^{i}$.
\end{enumerate}
\end{lemma}
\begin{proof}(1) follows from the Clebsch-Gordon rule \cite{J}:
 $$L(m)\otimes L(n)\cong L(m+n)\oplus L(m+n-2)\oplus\dots \oplus L(m-n),\ \ m\geq n. $$

(2) We can check that  $\aligned E(X\otimes v_i )=(E\cdot X)\otimes v_i+(K\cdot X)\otimes Ev_i=0,\endaligned$
$$\aligned & E([i] Y\otimes v_i- q^{-1}X\otimes Fv_i )\\
&=[i] (E\cdot Y)\otimes v_i+  [i] (K\cdot Y)\otimes Ev_i- q^{-1}(E\cdot X)\otimes Fv_i- q^{-1} (K\cdot X)\otimes EFv_i\\
&=  [i] X\otimes v_i- X\otimes EFv_i=0.
\endaligned$$ Then (2) holds.

(3) follows from (1) and (2).
\end{proof}

\begin{lemma}\label{LA1}For any module $V\in\mathcal{A}$, $\theta_i\in \Hom_{U_q(\sl_2)}(L(i), V)$,  the following  map
$$\aligned \tau(\theta_i):L(1)\otimes L(i) & \longrightarrow V \\
(aX+bY)\otimes v & \mapsto  (aX+bY)\theta_i(v)\\
\endaligned$$
 a $U_q(\sl_2)$-module homomorphism, where $\ a,b\in\C, v\in L(i)$.
\end{lemma}
\begin{proof}
For $a,b\in\C, v\in L(i)$,  we can check that
$$\aligned &\tau(\theta_i)\Big(K\big((aX+bY)\otimes v\big)\Big )\\
=&\tau(\theta_i)\Big((aK\cdot X+bK\cdot Y)\otimes Kv\Big )\\
=& aqXKv+bq^{-1}YKv\\
=& aKX v+bKYv \\
=& K\tau(\theta_i)\Big((aX+bY)\otimes v\Big ),
\endaligned$$
$$\aligned
 &\tau(\theta_i)\Big(E\big((aX+bY)\otimes v\big)\Big )\\
=&\tau(\theta_i)\Big((aK\cdot X+bK \cdot Y)\otimes Ev\Big )+\tau(\theta_i)\Big((aE \cdot X+bE\cdot Y)\otimes v \Big )\\
=& \tau(\theta_i)\Big((aqX+bq^{-1} Y)\otimes Ev\Big )+\tau(\theta_i)\Big(bX\otimes v \Big )\\
=& aqXEv+bq^{-1}YEv+bXv\\
=& aEX v+bEYv \\
=& E\tau(\theta_i)\Big((aX+bY)\otimes v\Big ),
\endaligned.$$
$$\aligned
 &\tau(\theta_i)\Big(F\big((aX+bY)\otimes v\big)\Big )\\
=&\tau(\theta_i)\Big((aF\cdot X+bF\cdot Y)\otimes K^{-1}v\Big )+\tau(\theta_i)\Big((aX+bY)\otimes Fv \Big )\\
=& \tau(\theta_i)\Big(aY\otimes K^{-1}v\Big )+\tau(\theta_i)\Big((aX+bY)\otimes Fv \Big )\\
=& aYK^{-1}v+aXFv+bYFv\\
=& aFX v+bFYv \\
=& F\tau(\theta_i)\Big((aX+bY)\otimes v\Big ).
\endaligned.$$

So $\tau(\theta_i)$ is a $U_q(\sl_2)$-module homomorphism.

\end{proof}
Consider the following quiver.
\begin{displaymath}
\mathbf{Q}_{\infty}:\quad\,\,
\xymatrix{
\mathtt{0}\ar@/^/[rr]^{a_0}&& \mathtt{1}\ar@/^/[ll]^{b_0}
\ar@/^/[rr]^{a_1}&& \mathtt{2}\ar@/^/[ll]^{b_1}
\ar@/^/[rr]^{a_2}&& \ldots\ar@/^/[ll]^{b_2}
}
\end{displaymath}
The following theorem is inspired by the ideas in \cite{DLMZ, Mak}.
\begin{theorem}
\label{thm13}The category $\mathcal{A}$ is equivalent to the category $\mathcal{B}$
of finite dimensional representations for the quiver $\mathbf{Q}_{\infty}$
satisfying the following condition:
\begin{equation} \label{quiverrealtion} b_0a_0=0,\ \ a_ib_i=b_{i+1}a_{i+1}, \ i\in\Z_+.\end{equation}
\end{theorem}
\begin{proof}
By  Lemma~\ref{zuigaoquan},  we can define $U_q(\sl_2)$-module homomorphisms
$$t_{i+1}: L(i+1) \rightarrow L(1)\otimes L(i), \ \ \ \  t'_{i-1}: L(i-1)\rightarrow L(1)\otimes L(i)$$ such that
$$t_{i+1}(v_{i+1})=X\otimes v_i, ,\ \ \ \ t'_{i-1}(v_{i-1})= [i] Y\otimes v_i- q^{-1}X\otimes Fv_i ,$$
where each $v_i$ is a fixed  highest weight vector of $L(i)$.

We will prove the theorem in three steps.

{\bf Step 1.}  We define a functor $F$ from $\mathcal{A}$ to $\mathcal{B}$.  Let $V$ be a $U_q(\ms)$-module which belongs to  $\mathcal{A}$.

(1) For every  $i$, we  can associate it with a vector space $V_i:=\Hom_{U_q(\sl_2)}(L(i), V)$.

(2) For arrows~$a_i , b_{i}$, we can define linear maps $V(b_{i}): V_{i+1}\rightarrow V_{i},  V(a_{i}): V_i\rightarrow V_{i+1}$  as follows:
$$V(b_{i})(\theta_{i+1})=\tau(\theta_{i+1})t'_{i}, \ \ \ \, V(a_{i})(\theta_i)=\tau(\theta_i)t_{i+1}.$$
Next we check that: $$V(b_0)V(a_0)=0,\ \ \ \ \, V(a_{i-1})V(b_{i-1})=V(b_{i})V(a_{i}), i\in\N.$$

For $\theta_0\in V_0, v\in L(0)$, from $F v=0$,  we have 
$$\aligned\Big(V(b_0)V(a_0)(\theta_0)\Big) (v)&=\tau\big(\tau(\theta_0)\big)(Y\otimes X\otimes v -q^{-1}X\otimes X\otimes Fv- q^{-1}X\otimes Y\otimes v)\\
&=\tau\big(\tau(\theta_0)\big)(Y\otimes X\otimes v - q^{-1}X\otimes Y\otimes v)\\
&=YX \theta_0(v)-q^{-1}XY\theta_0(v)=0 .\endaligned $$

For $\theta_i\in V_i$,  from
$XY=qYX, [i+1]q^{-1}-q^{-i-1}=[i]$,
 we have
$$\aligned& \Big(V(a_{i-1})V(b_{i-1})(\theta_i)\Big) (v_i) \\ &=\tau\big(\tau(\theta_i)\big)\big( [i] X\otimes Y\otimes v_i- q^{-1}X\otimes X\otimes Fv_i\big)\\
&= [i] X Y\theta_i(v_i)- q^{-1}X X F\theta_i(v_i)\\
&=[i+1]Y X\theta_i(v_i)-q^{-1}X X F\theta_i(v_i)-q^{-i-1}X Y\theta_i( v_i)\\
&=\tau\big(\tau(\theta_i)\big)\big( [i+1]Y\otimes X\otimes v_i-q^{-1}X\otimes X\otimes Fv_i-q^{-i-1}X\otimes Y \otimes v_i\big)\\
&= \Big(V(b_i)V(a_i)(\theta_i)\Big) (v_i).
 \endaligned $$

By the fact that the $U_q(\sl_2)$-module $L(i)$ is generated by  $v_i$ and
$V(b_i)V(a_i)(\theta_i)$, $V(a_{i-1})V(b_{i-1})(\theta_i)$ are $U_q(\sl_2)$-module homomorphisms,
we see that $V(b_i)V(a_i)(\theta_i)=V(a_{i-1})V(b_{i-1})(\theta_i)$.

Thus $(V_i, V(a_{i}), V(b_{i}), i\in\Z_+)$ is a representation of  $\mathbf{Q}_{\infty}$ satisfying  the relation (\ref{quiverrealtion}).

(3) We define a functor $F$ from $\mathcal{A}$ to $\mathcal{B}$.  For $V,W\in\text{Obj}(\mathcal{A}), f\in\Hom_{U_q(\ms)}(V, W)$,  define
$$F(V)=(V_i, V(a_{i}), V(b_{i}), i\in\Z_+),$$
$$F(f)=(f^*_i: V_i\rightarrow W_i,\ i\in \Z_+), $$
where $f^*_i$ satisfies $f^*_i(\theta_i)=f\theta_i, \theta_i\in V_i$.

We check the following diagram
\begin{displaymath}
\xymatrix{
  V_{i-1} \ar[d]_{f^*_{i-1}}
                & V_{i} \ar[d]^{f^*_{i}} \ar[l]_{V(b_{i-1}) }\\
  W_{i-1}
                & W_{i}  \ar[l]_{W(b_{i-1}) }        ,}\ \ \ \ \
\xymatrix{
  V_i \ar[d]_{f^*_i} \ar[r]^{V(a_{i})}
                & V_{i+1} \ar[d]^{f^*_{i+1}}  \\
  W_i \ar[r]_{W(a_{i})}
                & W_{i+1}           }
\end{displaymath} is commutative.

Since $\tau(\theta_i)$ is a $U_q(\sl_2)$-module homomorphism,
$$\aligned & W(b_{i-1})f^*_i(\theta_i)=W(b_{i-1})(f\theta_i)\\
=& \tau(f\theta_i )t'_{i-1}= f\tau(\theta_i )t'_{i-1}=f^*_{i-1}V(b_{i-1})(\theta_i),
\endaligned$$
and $$\aligned & W(a_i)f^*_i(\theta_i)=W(a_i)(f\theta_i)\\
=& \tau(f\theta_i )t_{i+1}= f\tau(\theta_i )t_{i+1}=f^*_{i+1}V(a_i)(\theta_i).
\endaligned$$

So $$f^*_iV(b_{i})=W(b_i)f^*_{i+1},f^*_{i+1}V(a_{i})=W(a_i)f^*_i.$$

Therefore, $F(f)=(f^*_i: V_i\rightarrow W_i,\ i\in \Z_+)$ is a morphism from the representation $F(V)$  to $F(W)$.

{\bf Step 2.} For a representation $ (V_i, V(a_{i}), V(b_{i}), i\in\Z_+)$ of the
 quiver $\mathbf{Q}_{\infty}$ satisfying  the relation (\ref{quiverrealtion}),
there is a $U_q(\ms)$-module~$V$ such that $F(V)\cong (V_i, V(a_{i}), V(b_{i}), i\in\Z_+)$.

 Let $V=\bigoplus_{i\in\Z_+} V_i\otimes L(i)$.  Next we define the action of
$U_q(\ms)$ on $V$.  Since $v_i$ is a highest weight vector of $L(i)$,  $F^s v_i, s\in\Z_+$ is a basis of $L(i)$.
For $ u\in U_q(\sl_2), \theta_i\in V_i, s\in\Z_+$,
define \begin{equation}\label{sa}\begin{aligned}u(\theta_i\otimes F^s v_i)&=\theta_i\otimes u F^s v_i,\\
  X(\theta_i\otimes F^s v_i)
& =\frac{q^s[i+1-s]}{[i+1]}V(a_i)(\theta_i)\otimes F^s v_{i+1} - \frac{q^{s-i}[s]}{[i+1]}V(b_{i-1})(\theta_i)\otimes F^{s-1} v_{i-1},\\
Y(\theta_i\otimes F^s v_i)& = \frac{q}{[i+1]}V(b_{i-1})(\theta_i)\otimes F^s v_{i-1}+\frac{1}{[i+1]}V(a_i)(\theta_i)\otimes F^{s+1} v_{i+1}.
  \end{aligned}\end{equation}
  The reason for defining the  action of
$U_q(\ms)$ on $V$ by (\ref{sa}) coming from the definitions of  $V(a_{i}), V(b_{i})$  in Step 1 .
 We can check  that the action (\ref{sa}) indeed defines a $U_q(\ms)$-module through a little cumbersome calculation.
 For the verification process, one can see the appendix of the present paper.

From the definition of $F$, we can see that $F(V)\cong (V_i, V(a_{i}), V(b_{i}), i\in\Z_+)$.

{\bf Step 3.} The functor $F$ is completely faithful, i.e., the map
 $$F_{V,W}:\Hom_{\mathcal{A}}(V,W)\rightarrow \Hom_{\mathcal{B}}(F(V),F(W))$$ is a bijection.

If $f\in\Hom_{U_q(\ms)}(V, W)$ satisfying that $F_{V,W}(f)=0$,
then for any $\theta_i\in \Hom_{U_q(\sl_2)}(L(i), V)$, $i\in\Z_+$, we have:~$f\theta_i=0$.  Since $V$ is a sum of simple submodules $L(i)$.
So
$f=0$, and $F_{V,W}$ is injective.

From the completely reducibility of finite dimensional $U_q(\sl_2)$-modules, $$V=\sum_{i\in \Z_+}\sum_{\theta_\in V_i} \theta_i (L(i)).$$
For $g=(g_i: V_i\rightarrow W_i,\ i\in \Z_+)\in \Hom_{\mathcal{B}}(F(V),F(W))$,  we define $f:V\rightarrow W$ as follows:
$$f\big(\theta_i(w_i)\big)= g_i(\theta_i)(w_i), \ \theta_i\in V_i,\ w_i\in L(i).$$
Since $\theta_i, g_i(\theta_i)$ is a $U_q(\sl_2)$-module homomorphism,  for $u\in U_q(\sl_2)$, we have
$$\aligned   f\big(u(\theta_i(w_i))\big)=f\big(\theta_i(uw_i)\big)
= g_i(\theta_i)(uw_i)= u(g_i(\theta_i))(w_i)=uf\big(\theta_i(w_i)\big).  \endaligned$$
At the same time, using the following commutative diagram:
\begin{displaymath}
\xymatrix{
  V_{i-1} \ar[d]_{g_{i-1}}
                & V_{i} \ar[d]^{g_i} \ar[l]_{V(b_{i-1}) }\\
  W_{i-1}
                & W_{i}  \ar[l]_{W(b_{i-1}) }        },\ \ \ \ \
\xymatrix{
  V_i \ar[d]_{g_i} \ar[r]^{V(a_{i})}
                & V_{i+1} \ar[d]^{g_{i+1}}  \\
  W_i \ar[r]_{W(a_{i})}
                & W_{i+1}           }
\end{displaymath} we obtain that
$$\aligned  f\big(X\theta_i(v_i)\big)=& f \Big( (V(a_i)(\theta_i))( v_{i+1})\Big)= g_{i+1}V(a_i)(\theta_i)( v_{i+1})\\
=& W(a_i)g_i(\theta_i)( v_{i+1}) =Xf\big(\theta_i(v_i) \big)\endaligned$$
and $$\begin{aligned}
f\big(Y\theta_i(v_i)\big)=& f\big(\frac{q}{[i+1]}V(b_{i-1})(\theta_i)(v_{i-1})+\frac{1}{[i+1]}V(a_i)(\theta_i)(F v_{i+1})\big)\\
= & \frac{q}{[i+1]}g_{i-1}V(b_{i-1})(\theta_i)(v_{i-1})+\frac{1}{[i+1]}g_{i+1}V(a_i)(\theta_i)(F v_{i+1})\\
=& \frac{q}{[i+1]}W(b_{i-1})g_{i}(\theta_i)(v_{i-1})+\frac{1}{[i+1]}W(a_i)g_{i}(\theta_i)(F v_{i+1})\\
=& Y g_{i}(\theta_i)(v_i)=Yf\big(\theta_i(v_i)\big).
  \end{aligned}$$

So $f$ is a $U_q(\ms)$-module homomorphism such that $F(f)=g$.

Therefore, $F$ is an equivalence.
\end{proof}
Let $\C\langle x_1, x_2\rangle$ be the free associative algebra over $\C$ generated by two variables $x_1, x_2$.  Recall that an abelian category $\mathcal{C}$ is wild if there exists an exact functor from the category of representations
 of the algebra $\C\langle x_1, x_2\rangle$ to $\mathcal{C}$ which preserves indecomposability and takes nonisomorphic modules to
 nonisomorphic ones, see Definition 2 in \cite{Mak}.

By \cite{DLMZ}, the category $\mathcal{B}$  for the quiver $\mathbf{Q}_\infty$ is wild. Hence we have the following corollary.
\begin{corollary}
The representation type of the category $\mathcal{A}$ is wild.
\end{corollary}

Let $\mathcal{O}[0]$ be the full subcategory $\mathcal{O}$ consisting of modules with locally nilpotent action of $Z$.  Since
$\mathcal{A}$ is a subcategory of $\mathcal{O}[0]$ , $\mathcal{O}[0]$ is also wild.

\section{Appendix}

In this appendix, we check that the  action (\ref{sa}) indeed defines a $U_q(\ms)$-module.
Let $v_i$ be  a fixed  highest weight vector of $L(i)$,  $\theta_i\in V_i, s\in\Z_+$.

From the action of
$U_q(\sl_2)$ on $V$,  it suffices to check  the
 relations $XY=qYX$ and the relations between $H_q$ and $U_q(\sl_2)$.

Firstly it is easy to see that $FY(\theta_i\otimes F^s v_i)= YF(\theta_i\otimes F^s v_i)$ from $FY=YF$.

Next we can compute that:
$$\aligned XY(\theta_i\otimes F^s v_i)
 =& \frac{q}{[i+1]}X \Big( V(b_{i-1})(\theta_i)\otimes F^s v_{i-1}\Big)\\
 &+\frac{1}{[i+1]}X \Big( V(a_i)(\theta_i)\otimes F^{s+1} v_{i+1}\Big)\\
= &\frac{q^{s+1}[-s+i]}{[i][i+1]}V(a_{i-1})V(b_{i-1})(\theta_i)\otimes F^s v_{i}\\
& - \frac{q^{s-i+2}[s]}{[i+1][i]}V(b_{i-2})V(b_{i-1})(\theta_i)\otimes F^{s-1} v_{i-2}\\
&- \frac{q^{s-i}[s+1]}{[i+1][i+2]}V(b_{i}) V(a_i)(\theta_i)\otimes F^{s} v_{i}\\
&+ \frac{q^{s+1}[-s+i+1]}{[i+1][i+2]}V(a_{i+1}) V(a_i)(\theta_i)\otimes F^{s+1} v_{i+2},
\endaligned$$
 $$\aligned
&\\
   YX(\theta_i\otimes F^s v_i)
  =& \frac{q^{s}[-s+i+1]}{[i+1]}Y\Big( V(a_i)(\theta_i)\otimes F^s v_{i+1}\Big) \\
  &- \frac{q^{s-i}[s]}{[i+1]}Y\Big(V(b_{i-1})(\theta_i)\otimes F^{s-1} v_{i-1}\Big)\\
  =& \frac{q^{s+1}[-s+i+1]}{[i+1][i+2]}V(b_{i})V(a_i)(\theta_i)\otimes F^s v_{i}\\
  &+\frac{q^{s}[-s+i+1]}{[i+1][i+2]}V(a_{i+1})V(a_i)(\theta_i)\otimes F^{s+1} v_{i+2}\\
   &  - \frac{q^{s-i+1}[s]}{[i+1]}\frac{1}{[i]}V(b_{i-2})V(b_{i-1})(\theta_i)\otimes F^{s-1} v_{i-2} \\
   &- \frac{q^{s-i}[s]}{[i+1]}\frac{1}{[i]}V(a_{i-1})V(b_{i-1})(\theta_i)\otimes F^{s} v_{i}.
  \endaligned$$
Then using $ V(a_{i-1})V(b_{i-1})=V(b_{i})V(a_{i})$ and
  $$ \frac{q^{s+1}[-s+i]}{[i][i+1]}- \frac{q^{s-i}[s+1]}{[i+1][i+2]}=q\Big(\frac{q^{s+1}[-s+i+1]}{[i+1][i+2]}- \frac{q^{s-i}[s]}{[i+1]}\frac{1}{[i]}\Big),$$
we have
  ~$ XY(\theta_i\otimes F^s v_i)=q YX(\theta_i\otimes F^s v_i) $.

Furthermore, from $EF^sv_i=[s][i+1-s]F^{s-1}v_i$ , we have
  $$\aligned & EX(\theta_i\otimes F^s v_i) \\
  =& \frac{q^{s}[i+1-s]}{[i+1]} V(a_i)(\theta_i)\otimes EF^s v_{i+1} - \frac{q^{s-i}[s]}{[i+1]}V(b_{i-1})(\theta_i)\otimes E F^{s-1} v_{i-1}\\
  =&  [s][i+2-s]\frac{q^{s}[i+1-s]}{[i+1]}V(a_i)(\theta_i)\otimes F^{s-1}v_{i+1}\\
  & - [s-1][i+1-s]\frac{q^{s-i}[s]}{[i+1]}V(b_{i-1})(\theta_i)\otimes F^{s-2}v_{i-1},
    \endaligned$$
  $$\aligned
XE(\theta_i\otimes F^s v_i)=& X(\theta_i\otimes EF^s v_i)\\
  =& [s][i+1-s]  X(\theta_i\otimes F^{s-1}v_i)\\
  = &[s][i+1-s] \frac{q^{s-1}[i+2-s]}{[i+1]}V(a_i)(\theta_i)\otimes F^{s-1} v_{i+1}\\
  &  - [s][i+1-s] \frac{q^{s-i-1}[s-1]}{[i+1]}V(b_{i-1})(\theta_i)\otimes F^{s-2} v_{i-1}.
   \endaligned$$

Then $  EX(\theta_i\otimes F^s v_i)=q XE(\theta_i\otimes F^s v_i)$.

 Using $EF^s=F^s E+ [s]F^{s-1}\frac{q^{-(s-1)}K-q^{s-1}K^{-1}}{q-q^{-1}}$, we have
$$\aligned  EY(&\theta_i\otimes F^s v_i) \\
 =& \frac{q}{[i+1]}V(b_{i-1})(\theta_i)\otimes EF^s v_{i-1}+\frac{1}{[i+1]}V(a_i)(\theta_i)\otimes EF^{s+1} v_{i+1}\\
=& \frac{q[s][i-s]}{[i+1]}V(b_{i-1})(\theta_i)\otimes F^{s-1} v_{i-1}+\frac{[s+1][i-s+1]}{[i+1]}V(a_i)(\theta_i)\otimes F^{s} v_{i+1},
   \endaligned$$
   $$\aligned & (X+ q^{-1}YE)(\theta_i\otimes F^s v_i)\\
 =& \frac{q^s[i+1-s]}{[i+1]}V(a_i)(\theta_i)\otimes F^s v_{i+1} - \frac{q^{s-i}[s]}{[i+1]}V(b_{i-1})(\theta_i)\otimes F^{s-1} v_{i-1}\\
  & +\frac{[s][i+1-s]}{[i+1]}V(b_{i-1})(\theta_i)\otimes F^{s-1} v_{i-1}+\frac{q^{-1}[s][i+1-s]}{[i+1]}V(a_i)(\theta_i)\otimes F^{s} v_{i+1}. \endaligned$$
  Then from $[i+1-s]-q^{s-i}=q[i-s]$ and $q^s+q^{-1}[s]=[s+1]$, we have $$EY(\theta_i\otimes F^s v_i) = (X+ q^{-1}YE)(\theta_i\otimes F^s v_i).$$

    Next we have
  $$\aligned &  FX(\theta_i\otimes F^s v_i)
  = \frac{q^s[i+1-s]}{[i+1]}V(a_i)(\theta_i)\otimes F^{s+1} v_{i+1} - \frac{q^{s-i}[s]}{[i+1]}V(b_{i-1})(\theta_i)\otimes F^{s} v_{i-1},
\endaligned$$
       $$\aligned  &(YK^{-1}+XF)(\theta_i\otimes F^s v_i)\\
       = & \frac{ q^{2s-i+1} }{[i+1]}V(b_{i-1})(\theta_i)\otimes F^s v_{i-1}+\frac{ q^{2s-i} }{[i+1]}V(a_i)(\theta_i)\otimes F^{s+1} v_{i+1},\\
       & + \frac{q^{s+1}[i-s]}{[i+1]}V(a_i)(\theta_i)\otimes F^{s+1} v_{i+1} - \frac{q^{s+1-i}[s+1]}{[i+1]}V(b_{i-1})(\theta_i)\otimes F^{s} v_{i-1}.
       \endaligned$$
Using   $$\aligned-q^{s-i}&=q^{2s-i+1} -q^{s+1-i}[s+1], \\
q^s[i+1-s]&=q^{2s-i}+q^{s+1}[i-s],\endaligned$$ we have
$$FX(\theta_i\otimes F^s v_i)= (YK^{-1}+XF)(\theta_i\otimes F^s v_i).$$

\vspace{5mm}
\vspace{0.2cm}
\noindent G. Liu: School
of Mathematics and Statistics, Henan University, Kaifeng 475004, P.R. China. Email:
liugenqiang@amss.ac.cn
\vspace{0.2cm}

\noindent Y. Li: School of Mathematics and Statistics, Henan University, Kaifeng
475004, China. Email: 897981524@qq.com

\end{document}